%% file: simplifying_triangulations.tex
\documentclass[a4paper]{article}

\input{settings}

\title{Simplifying triangulations}
\author{Mark C. Bell\\
University of Illinois\\
\url{mcbell@illinois.edu}}

\begin{document}

\maketitle

\begin{abstract}
We give a new algorithm to simplify a given triangulation with respect to a given curve. The simplification uses flips together with powers of Dehn twists in order to complete in polynomial time in the bit-size of the curve.
\end{abstract}


\ccode{57M20}  

\section{Introduction}

Let $S$ be an (orientable) punctured surface and let $\zeta = \zeta(S) \defeq -3 \chi(S)$. We will assume that $S$ is sufficiently complex that $\zeta \geq 3$ and so $S$ can be decomposed into an (ideal) triangulation. Any such triangulation of $S$ has exactly $\zeta$ edges.

A curve $\gamma$ on $S$ may appear extremely complicated from the point of view of a triangulation. However there is always a triangulation $\calT$ in which $\intersection(\gamma, \calT) \leq 2 \zeta$. Such a triangulation, which we refer to as \emph{$\gamma$--simple}, is extremely useful for performing calculations with. To give just a few examples, if $\gamma$ is given on a $\gamma$--simple triangulation then it is straightforward to:
\begin{itemize}
\item determine its topological type,
\item compute its algebraic intersection number with an edge, and  
\item verify that it is connected.
\end{itemize}

The aim of this paper is to show that a small collection of basic moves can be used to rapidly transform a given triangulation into a $\gamma$--simple one. The key result for achieving this is:

\begin{restate}{Theorem}{thrm:main}
Let $D \defeq 80 \zeta B (10B + 1)^C$ where $B \defeq 5^{2\zeta}$ and $C \defeq 2^{2\zeta}$. If $\intersection(\gamma, \calT) > D$ then there is a triangulation $\calT'$ such that either:
\begin{itemize}
\item $\calT$ and $\calT'$ differ by a flip, or
\item $\calT' = T_\delta^k(\calT)$ where $|k| \leq \intersection(\gamma, \calT)$ and $\intersection(\delta, \calT) \leq 2 \zeta$
\end{itemize}
and $\intersection(\gamma, \calT') \leq (1 - 1/D) \intersection(\gamma, \calT)$. \qed
\end{restate}


Using this, as we can reduce $\intersection(\gamma, \calT)$ by a definite fraction by performing flips and Dehn twists, we can convert $\calT$ to a $\gamma$--simple one in only $O(\log(\intersection(\gamma, \calT)))$ such moves. This result mimics several similar simplification results in other models of curves on surfaces. For example in:
\begin{itemize}
\item interval permutations by Agol, Hass and Thurston \cite[Section~4]{AHT},
\item the street complex by Erickson and Nayyeri \cite{EricksonNayyeri},
\item straight line programs by Schaefer, Sedgwick and \v{S}tefankovi\v{c} \cite{SchaeferSedgwick}, and
\item Dynnikov coordinates on a punctured disk \cite{DynnikovBraids}.
\end{itemize}

\section{Flips}
\label{sec:flips}

The first basic operation that we will consider in order to produce a simpler triangulation with respect to $\gamma$ is a \emph{flip}.

We say that an edge of $\calT$ is \emph{flippable} if it is contained in two distinct triangles. If $e$ is such an edge then we may flip it to obtain a new triangulation $\calT'$ as shown in Figure~\ref{fig:flip}.

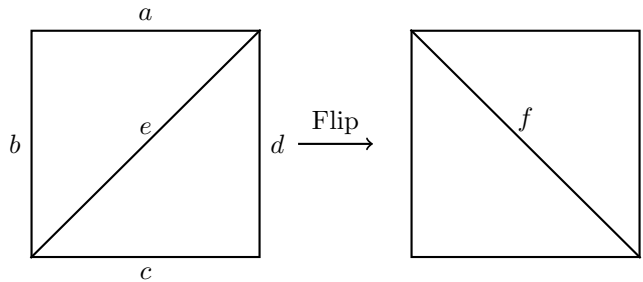
\begin{figure}[ht]
\centering
\input{tikz_flip}
\caption{Flipping an edge of a triangulation.}
\label{fig:flip}
\end{figure}

The number of intersections between $\gamma$ and the new edge $f$ is exactly determined by the number of intersections between $\gamma$ and the neighbouring edges of $e$.

\begin{proposition}[{\cite[Page~30]{MosherFoliations}}]
\label{prop:flip_intersection_laminations}
Suppose that $\gamma$ is a curve and $e$ is a flippable edge of a triangulation $\calT$ as shown in Figure~\ref{fig:flip} then
\[ \intersection(\gamma, f) = \max(\intersection(\gamma, a) + \intersection(\gamma, c), \intersection(\gamma, b) + \intersection(\gamma, d)) - \intersection(\gamma, e). \inlineQED \]
\end{proposition}

\subsection{The flip graph}

The \emph{flip graph} $G = G(S)$ is the graph with a vertex for each triangulation of $S$ where two vertices are connected via an edge of length $1$ if they differ by a flip. The flip graph is connected \cite{HatcherTriangulations} \cite[Page~36]{MosherFoliations} and so we may use a sequence of flips to convert $\calT$ to a $\gamma$--simple triangulation. To help us find such a sequence we recall the following lemma:

\begin{lemma}[{\cite[Lemma~2.4.3]{BellThesis}}]
\label{lem:drop_intersection}
If $\intersection(\gamma, \calT) > 2\zeta$ then there is an edge of $\calT$ which can be flipped in order to reduce the intersection number. \qed
\end{lemma}

Similar results are also known for other measures of the complexity of $\gamma$ \cite{PennerTropical}\cite[Page~39]{MosherFoliations}.
We may use Lemma~\ref{lem:drop_intersection} repeatedly to monotonically reduce $\intersection(\gamma, \calT)$ until we reach a $\gamma$--simple triangulation and so deduce:

\begin{corollary}
For each $\calT \in G$ and curve $\gamma$ there is a $\gamma$--simple triangulation $\calT' \in G$ such that $d(\calT, \calT') \in O(\intersection(\gamma, \calT))$. \qed
\end{corollary}

Unfortunately there are cases in which at least $\intersection(\gamma, \calT)$ flips are required in order to obtain a $\gamma$--simple triangulation. For example, on the triangulation of the once punctured torus shown in Figure~\ref{fig:slow_flips} the curve of slope $k$ has geometric intersection number $\approx k$ while the nearest $\gamma$--simple triangulation is $\approx k$ away.

\begin{figure}[ht]
\centering
\input{tikz_slow}
\caption{This triangulation is far from a $\gamma$--simple one in $G$.}
\label{fig:slow_flips}
\end{figure}
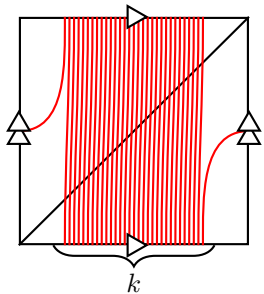

Such examples arise by performing large powers of Dehn twists. In the next section we will show that in fact these twists are the only way to create such an obstruction.

\section{Twists}

To deal with triangulations which need a large number of flips in order to simplify them we introduce a second type of move: the \emph{Dehn twist} $T_\delta^k$ \cite[Chapter~3]{FM}. This move cuts the surface open along the curve $\delta$ and rotates one of the boundary components $k$ times to the right (or $|k|$ times to the left if $k$ is negative) before regluing the boundary components together. We will show that if flips cannot decrease $\intersection(\gamma, \calT)$ by a definite fraction then a power of a Dehn twist can.

To do this, suppose that $\calT$ is a fixed triangulation. Suppose that $\gamma$ is a fixed curve and assume that flipping any edge of $\calT$ reduces $\intersection(\gamma, \calT)$ by at most $m$. Fix $\emax$ to be an edge of $\calT$ which meets $\gamma$ the most, that is, such that
\[ E \defeq \intersection(\gamma, \emax) \geq \intersection(\gamma, e) \]
for every edge $e$ of $\calT$. Additionally fix a coorientation $\emax$, that is, a choice of unit normal vector field to $\emax$.

Abusing notation slightly, let us think of $\gamma$ as a representative of its isotopy class which is in minimal position with respect to $\calT$. Let $P \defeq \gamma \cap \calT$.

\subsection{Insulation}

\begin{definition}
Suppose that $p \in P$ lies on the edge $e$ of $\calT$. Then $p$ is \emph{$k$--insulated} if each component of $e - p$ contains at least $k$ other points in $P$. That is, if looking along $e$ there are at least $k$ points in $P$ on either side of $p$. For example, see Figure~\ref{fig:insulation}.
\end{definition}

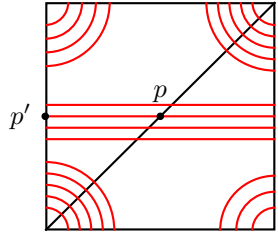
\begin{figure}[ht]
\centering
\input{tikz_insulation}
\caption{A $6$--insulated point $p \in P$ and adjacent $5$--insulated point $p' \in P$.}
\label{fig:insulation}
\end{figure}

We say that $p, p' \in P$ are \emph{adjacent} if they appear consecutively along $\gamma$. For example, again see Figure~\ref{fig:insulation}. The key property of insulation is that it only decays slightly when we move to an adjacent point.

\begin{lemma}
\label{lem:adjacent_insulation}
Suppose that $p \in P$ and that $p' \in P$ is an adjacent point. If $p$ is $k$--insulated then $p'$ is $(5k - (2E + m - 2))$--insulated.
\end{lemma}

\begin{proof}
For convenience we will use the notation $\mathbf{x} \defeq \intersection(\gamma, x)$ here.

We begin by considering the case in which $p$ lies on a flippable edge $e$. Without loss of generality, following the notation of Figure~\ref{fig:flip}, we may assume that $\bi + \di \geq \ai + \ci$.

Recall that by Proposition~\ref{prop:flip_intersection_laminations} we have that
\[ \bi + \di - \ei = \ffi \geq \ei - m. \]
Now if $\bi < 2 \ei - (m + E)$ then
\[ E \geq \di \geq 2\ei - \bi - m > 2\ei - (2 \ei - (m + E)) - m = E, \]
a contradiction. By symmetry, the same inequality holds for $\di$ and so by combining this with the fact that $\bi, \di \leq E$ we have that
\begin{equation}
\label{eqn:insulation}
2 \ei - (m + E) \leq \bi, \di \leq E.
\end{equation}

\begin{figure}[ht]
\centering
\input{tikz_corners}
\caption{Curves in the corners of the square about $e$.}
\label{fig:corners}
\end{figure}
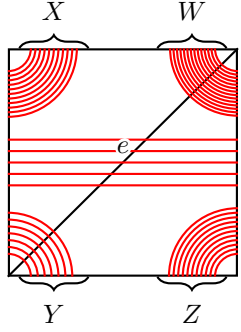

Let $W$, $X$, $Y$ and $Z$ be the number of times that $\gamma$ runs around each of the corners of the square about $e$, as shown in Figure~\ref{fig:corners}. Using this notation we then have that
\[ -m - 2(E - \ei) \leq \bi - \di = (X + Y) - (Z + W) \leq m + 2 (E - \ei) \]
by \eqref{eqn:insulation}. Additionally, by \eqref{eqn:insulation} we also have that
\[ -m \leq \ffi - \ei = (X + Z) - (W + Y) \leq 2 (E - \ei). \]
By adding and subtracting these inequalities we discover that
\[ |X - W| \leq m + 2 (E - \ei) \inlineand |Y - Z| \leq m + 2 (E - \ei). \]

Now let $e'$ be the edge containing $p'$. All of the points of $P$ on one component of $e - p$ are adjacent to points on one of the components of $e' - p'$. However, as $X \approx W$ and $Y \approx Z$, the other component of $e' - p'$ can only have $m + 2(E - \ei)$ fewer points of $P$ than the other component of $e - p$. Hence $p'$ must be $(k - m - 2E + 2\ei)$--insulated.

Finally note that as $p$ is $k$--insulated we must have that $\ei \geq 2k + 1$. Therefore $p'$ is $(5k - (2E + m - 2))$--insulated as required.

On the other hand, suppose that $p$ lies on a non-flippable edge $e$. In this case $p'$ must lie on the bounding edge $e'$, as shown in Figure~\ref{fig:non_flip}.

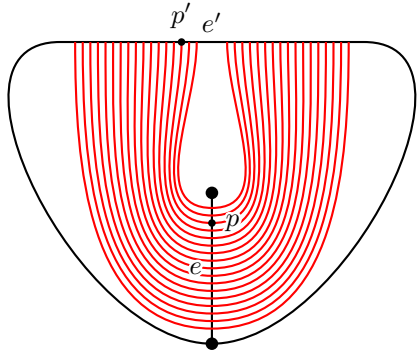
\begin{figure}[ht]
\centering
\input{tikz_non_flip}
\caption{If $p$ lies on a non-flippable edge $e$ then $p'$ is at least as insulated as $p$.}
\label{fig:non_flip}
\end{figure}

It follows immediately that if $p$ is $k$--insulated then $p'$ is $k$--insulated too. As $k \geq 5k - (2E + m - 2)$ the result also holds in this case.
\end{proof}

For convenience, let $A \defeq \frac{2E + m - 2}{4}$. The next corollary follows from solving the difference equation:
\[ k_{i+1} = 5k_{i} - 4A. \]

\begin{corollary}
\label{cor:adjacent_insulation}
Suppose that $p_0, p_1, \ldots$ is a sequence of pairwise adjacent points. If $p_0$ is $k$--insulated then $p_i$ is $(A - (A - k)5^i)$--insulated. \qed
\end{corollary}

\subsection{Blocks}

We focus on the intersections between $\gamma$ and $\emax$, so define $Q \defeq P \cap \emax$. The following definition will be used extensively throughout the remainder of this section:

\begin{definition}
The \emph{chain} of $q \in Q$ is the sequence of pairwise adjacent points $p_0, p_1, \ldots, p_{2 \zeta}$, where $p_0 = q$, emanating from $\emax$ in the direction of its coorientation.
\end{definition}

If $p_0, p_1, \ldots, p_{2\zeta}$ is the chain of $q \in Q$ then we refer to the edge that $p_i$ lies on together with the coorientation with which the chain meets that edge at $p_i$ as the \emph{type} of $p_i$. Of course, as there are only $2\zeta$ different types, there is always a pair of points in the chain of $q \in Q$ of the same type.

We partition the points of $Q$ into subsets $Q_1, Q_2, \ldots, Q_C$ called \emph{blocks}. The block that $q \in Q$ is contained in is determined by the sequence of types of the points in its chain. We note that the points in each block are consecutive along $\emax$ and that there are at most $C \defeq 2^{2\zeta}$ blocks. This bound can be seen from the fact that each block is determined by the sequence of $2 \zeta$ left or right turns made by a representative chain in it and can almost certainly be improved.


Again for convenience, let $B \defeq 5^{2\zeta}$.

\begin{proposition}
\label{prop:first_return}
Suppose that $q \in Q$ is $k$--insulated and lies in the block $Q_n$. If $Q_n$ contains more than
\[ E - 2(A - AB + Bk) \]
points then the lexicographically smallest pair $(i, j)$ such that $p_i$ and $p_j$ in the chain of $q$ have the same type is of the form $(0, j)$.
\end{proposition}

\begin{proof}
We will show the contrapositive. Suppose that $(i, j)$ is the lexicographically smallest pair such that $p_i$ and $p_j$ have the same type and that $i > 0$. Let $e$ be the edge containing $p_i$ and $p_j$.

\begin{figure}[ht]
\centering
\input{tikz_partition}
\caption{A block must be narrow if it meets itself but not on $\emax$.}
\label{fig:partition_chain}
\end{figure}
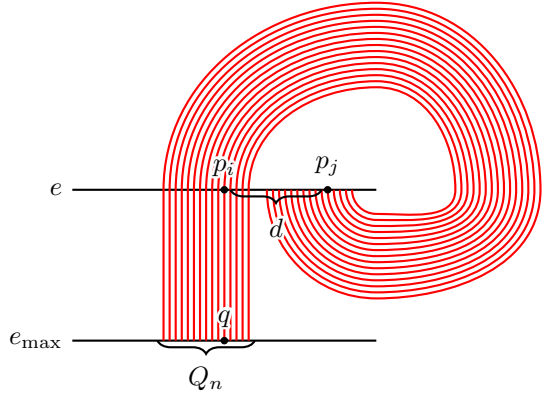

Suppose that there are $d$ points of $P$ between $p_i$ and $p_j$ along $e$. Note that, as shown in Figure~\ref{fig:partition_chain}, there are at most $d + 1$ points in $Q_n$. However, by Corollary~\ref{cor:adjacent_insulation}, $p_i$ and $p_j$ are both $(A - (A- k)B)$--insulated and so
\[ |Q_n| \leq d + 1 \leq E - 2(A - AB + Bk) \]
as required.
\end{proof}

We can now prove the main theorem. To ease notation in the statement and proof of the theorem, let $D \defeq 80 \zeta B (10B + 1)^C$, which depends only on $\chi(S)$.

\begin{theorem}
\label{thrm:main}
If $\intersection(\gamma, \calT) > D$ then there is a triangulation $\calT'$ such that either:
\begin{itemize}
\item $\calT$ and $\calT'$ differ by a flip, or
\item $\calT' = T_\delta^k(\calT)$ where $|k| \leq \intersection(\gamma, \calT)$ and $\intersection(\delta, \calT) \leq 2 \zeta$
\end{itemize}
and $\intersection(\gamma, \calT') \leq (1 - 1 / D) \intersection(\gamma, \calT)$.
\end{theorem}

\begin{proof}
Let $m \defeq \intersection(\gamma, \calT) / D$. We may assume that no flip reduces $\intersection(\gamma, \calT)$ by more than $m$ as otherwise we are done by performing that flip.

Without loss of generality we may also assume that the blocks $Q_1, \ldots, Q_C$ are ordered by insulation. That is,
\begin{itemize}
\item $Q_1$ is the block containing an innermost point $q_1$,
\item $Q_2$ is the block containing $q_2$, a point of highest insulation in $Q$ not in $Q_1$,
\item $Q_3$ is the block containing $q_3$, a point of highest insulation in $Q$ not in $Q_1 \cup Q_2$,
\item[] $\vdots$
\item $Q_C$ is the block containing $q_C$, a point of highest insulation in $Q$ not in $Q_1 \cup Q_2 \cup \cdots \cup Q_{C-1}$.
\end{itemize}

Now let $B_n \defeq 40 m B (10B + 1)^{n-1}$. If every $Q_n$ contains at most $B_n$ points then
\[ E \leq 40 m B (10B + 1)^C. \]
This cannot happen as it would mean that
\[ \intersection(\gamma, \calT) \leq \zeta E \leq 40 \zeta m B (10B + 1)^C = m D / 2 = \intersection(\gamma, \calT) / 2. \]
Therefore there is a smallest $n$ such that $Q_n$ contains more that $B_n$ points.

Now note that $q_n$, the point in $Q_n$ with maximal insulation that we chose above, is $k$--insulated where
\[ k \defeq \left \lfloor \frac{E-1}{2} \right \rfloor - (B_1 + \cdots + B_{n-1}) \geq \left \lfloor \frac{E - 1}{2} \right \rfloor - \frac{B_n}{10B}. \]
Let $p_0, p_1, \ldots, p_{2 \zeta}$ be the chain of $q_n$ and let $(i, j)$ be the lexicographically smallest pair such that $p_i$ and $p_j$ have the same type.
Now as
\[ E - 2(A - AB + Bk) \leq 2mB + B_n / 5 \leq B_n / 4 \]
we have that $i = 0$ by Proposition~\ref{prop:first_return}.
Therefore both $p_0$ and $p_j$ both lie on $\emax$ and there are at most $B_n / 4$ points between them along $\emax$.

Hence we let $\delta$ be the loop which runs parallel to $\gamma$ from $p_0$ round to $p_j$ and then connects back to $p_0$ by following along $\emax$.
As these points are so close and the block is so wide, the block and its image after it has been pushed along $\gamma$ share at least $3 B_n / 4$ points in $Q$.
It follows that by performing at most $B_n / 2$ Dehn twists along $\delta$ we can reduce $\intersection(\gamma, \calT)$ by at least $B_n / 2 > m$ as required.
\end{proof}

\subsection{The flip--twist graph}

To take into account this addition move, we introduce a modified version of the flip graph $G$ with additional edges. The \emph{flip--twist graph} $\calG = \calG(S)$ is the graph with a vertex for each triangulation of $S$ where $\calT$ and $\calT'$ are connected via:
\begin{itemize}
\item an edge of length $1$ if they differ by a flip, and
\item an edge of length $\log(\intersection(\delta, \calT) + k)$ if they differ by $T_\delta^k$.
\end{itemize}
These edge lengths are proportional to the computation complexity of performing these operations.

\begin{corollary}
For each $\calT \in \calG$ and curve $\gamma$ there is a $\gamma$--simple triangulation $\calT' \in \calG$ such that $d(\calT, \calT') \in O(\log(\intersection(\calT, \gamma))^2)$.
\end{corollary}

\begin{proof}
By applying Theorem~\ref{thrm:main} at most $\log(\intersection(\gamma, \calT)) / \log(N)$ times we can obtain a triangulation $\calT_0$ such that
\[ \intersection(\gamma, \calT_0) \leq D \inlineand d(\calT, \calT_0) \leq \log(2 \zeta + \intersection(\gamma, \calT)) \log(\intersection(\gamma, \calT)) / \log(N). \]
Now by applying Lemma~\ref{lem:drop_intersection} at most $D$ times we can obtain a triangulation $\calT'$ such that
\[ \intersection(\gamma, \calT') \leq 2\zeta \inlineand d(\calT_0, \calT') \leq D. \]
The result then follows from the triangle inequality.
\end{proof}

Again the curves of slope $k$ on the once-punctured torus from the end of Section~\ref{sec:flips} attain this logarithmic distance bound in $\calG$. The curve of slope $k$ on the once-punctured torus, as shown in Figure~\ref{fig:slow_flips}, has geometric intersection number $\approx k$ and the distance to the nearest $\gamma$--simple triangulation in $\calG$ is $\approx \log(k)$.

\section{Multicurves and multiarcs}

We finish by highlighting that a version of Theorem~\ref{thrm:main} holds if $\gamma$ is a multicurve or even a multiarc. However in the argument of the theorem it is possible that the wide block that we find returns exactly to itself, that is, $p_j = p_0$. In this case $\delta$ is disjoint from $\gamma$ and so performing Dehn twists about it has no effect.

If such a situation occurs then one can remove the entire block, which is a simple closed curve meeting each edge at most twice with high multiplicity. Again this reduces the intersection number by at least $m$ as required. Repeating this process allows us to extract the isotopy classes of $\gamma$ along with their multiplicities. We can then analyse each in turn in order to compute, for example, the topological types present.

Furthermore, in the multiarc case it may be necessary to repeat the analysis of the chains of $\emax$ with its other coorientation. This is because a block may terminate into a puncture before we can follow its chain for the required $2 \zeta$ steps. However if it terminates in both directions then it is a short arc and we can remove it to simplify the situation before repeating the argument.

\begin{acknowledgements}
The author would like to thank Saul Schleimer and Richard Webb for helpful suggestions, in particular for improving the constants used in this result.

The author also acknowledges support from U.S. National Science Foundation grants DMS 1107452, 1107263, 1107367 ``RNMS: GEometric structures And Representation varieties'' (the GEAR Network).
\end{acknowledgements}

\bibliographystyle{plain}
\bibliography{bibliography}

\end{document}

%% file: settings.tex
\usepackage{amssymb, amsmath, amscd, amsfonts, amsthm}
\usepackage{colonequals}

\usepackage{microtype} 
\usepackage{pinlabel} 

\usepackage[
pdfborderstyle={},
pdfborder={0 0 0},
pagebackref,
pdftex]{hyperref}

\renewcommand*{\backref}[1]{}
\renewcommand*{\backrefalt}[4]{
  \ifcase #1 %
   [No citations.]%
  \or
   [#2]%
  \else
   [#2]%
  \fi
}

\newcommand{\calG}{\mathcal{G}}
\newcommand{\calT}{\mathcal{T}}

\newcommand{\defeq}{\colonequals}
\DeclareMathOperator{\intersection}{\iota}

\theoremstyle{plain}
\numberwithin{equation}{section} 
\newtheorem{theorem}[equation]{Theorem}
\newtheorem*{theorem*}{Theorem}
\newtheorem{corollary}[equation]{Corollary}
\newtheorem*{corollary*}{Corollary}
\newtheorem{lemma}[equation]{Lemma}
\newtheorem*{lemma*}{Lemma}
\newtheorem{proposition}[equation]{Proposition}
\newtheorem*{proposition*}{Proposition}

\theoremstyle{definition}
\newtheorem{definition}[equation]{Definition}
\newtheorem*{definition*}{Definition}
\newtheorem*{keywords}{keywords}
\newtheorem*{acknowledgements}{Acknowledgements}

\newtheoremstyle{dotless}{}{}{}{}{\bfseries}{}{ }{}
\theoremstyle{dotless}
\newtheorem*{ccode}{\textbf{Mathematics Subject Classification (2010):}}

\newcommand{\fakeenv}{} 

\newenvironment{restate}[2]  
{ 
 \renewcommand{\fakeenv}{#2} 
 \theoremstyle{plain} 
 \newtheorem*{\fakeenv}{#1~\ref{#2}} 
 \begin{\fakeenv}
}
{
 \end{\fakeenv}
}

\newcommand{\inlineQED}{\pushQED{\qed} \qedhere \popQED}
\newcommand{\inlineand}{\quad \textrm{and} \quad}

\usepackage[outline]{contour}
\contourlength{0.1em}
\usepackage{color}
\usepackage{tikz}
\usetikzlibrary{calc}
\usetikzlibrary{decorations.pathreplacing}
\usetikzlibrary{shapes.geometric}

\newcommand{\ai}{\mathbf{a}}
\newcommand{\bi}{\mathbf{b}}
\newcommand{\ci}{\mathbf{c}}
\newcommand{\di}{\mathbf{d}}
\newcommand{\ei}{\mathbf{e}}
\newcommand{\ffi}{\mathbf{f}}

\newcommand{\emax}{e_{\textrm{max}}}

%% file: tikz_flip.tex
\begin{tikzpicture}[scale=2,thick]

\node (rect) at (-1.5, 0) [draw,minimum width=3cm,minimum height=3cm] {};
\node (rect2) at (1, 0) [draw,minimum width=3cm,minimum height=3cm] {};

\draw (rect.south west) -- node [above] {$e$} (rect.north east);
\draw (rect2.north west) -- node [above, yshift=1] {$f$} (rect2.south east);

\node (a) at (rect.north) [anchor=south] {$a$};
\node (b) at (rect.west) [anchor=east] {$b$};
\node (c) at (rect.south) [anchor=north] {$c$};
\node (d) at (rect.east) [anchor=west] {$d$};

\draw [thick,->] ($(rect.east)!0.25!(rect2.west)$) -- node[above] {Flip} ($(rect.east)!0.75!(rect2.west)$);
\end{tikzpicture}

%% file: tikz_slow.tex
\begin{tikzpicture}[scale=2,thick]

\tikzset{dot/.style={shape=circle,fill=black,scale=0.3}}
\tikzset{tri/.style={draw,scale=0.5,fill=white,regular polygon,regular polygon sides=3}}

\node (rect) at (-1.5, 0) [draw,minimum width=3cm,minimum height=3cm] {};

\draw (rect.south west) -- (rect.north east);

\draw [red] ($(rect.north west)!0.2!(rect.north east)$) to [out=270,in=0] (rect.west);
\draw [red] ($(rect.south east)!0.2!(rect.south west)$) to [out=90,in=180] (rect.east);

\foreach \i in {0.2,0.22,...,0.8}
	\draw [red] let \n1 = {\i+0.02} in ($(rect.south west)!\i!(rect.south east)$) to [out=90,in=270] ($(rect.north west)!\n1!(rect.north east)$);

\node [tri, rotate=30] at (rect.south) {};
\node [tri, rotate=30] at (rect.north) {};
\node [tri, yshift=-5] at (rect.west) {};
\node [tri, yshift=5] at (rect.west) {};
\node [tri, yshift=-5] at (rect.east) {};
\node [tri, yshift=5] at (rect.east) {};

\draw [decorate,decoration={brace,amplitude=8pt,mirror}] ($(rect.south west)!0.15!(rect.south east)$) -- ($(rect.south west)!0.85!(rect.south east)$) node [midway,yshift=-5mm] {$k$};

\end{tikzpicture}

%% file: tikz_insulation.tex
\begin{tikzpicture}[scale=2,thick]

\tikzset{dot/.style={shape=circle,fill=black,scale=0.3}}

\node (rect) at (-1.5, 0) [draw,minimum width=3cm,minimum height=3cm] {};

\draw (rect.south west) -- (rect.north east);

\foreach \i in {0.1,0.16,...,0.3}
	\draw [red] ($(rect.north west)!\i!(rect.south west)$) to [out=0,in=270] ($(rect.north west)!\i!(rect.north east)$);
\foreach \i in {0.1,0.15,...,0.3}
	\draw [red] ($(rect.south west)!\i!(rect.north west)$) to [out=0,in=90] ($(rect.south west)!\i!(rect.south east)$);
\foreach \i in {0.1,0.15,...,0.3}
	\draw [red] ($(rect.north east)!\i!(rect.north west)$) to [out=270,in=180] ($(rect.north east)!\i!(rect.south east)$);
\foreach \i in {0.1,0.17,...,0.3}
	\draw [red] ($(rect.south east)!\i!(rect.south west)$) to [out=90,in=180] ($(rect.south east)!\i!(rect.north east)$);

\foreach \i in {0.45,0.5,0.55,0.6}
	\draw [red] ($(rect.north west)!\i!(rect.south west)$) to [out=0,in=180] ($(rect.north east)!\i!(rect.south east)$);

\node [dot] at (rect.center) [label={$p$}] {};
\node [dot] at (rect.west) [label=left:{$p'$}] {};

\end{tikzpicture}

%% file: tikz_corners.tex
\begin{tikzpicture}[scale=2,thick]

\node (rect) at (-1.5, 0) [draw,minimum width=3cm,minimum height=3cm] {};

\draw (rect.south west) -- (rect.north east);

\foreach \i in {0.1,0.12,...,0.3}
	\draw [red] ($(rect.north west)!\i!(rect.south west)$) to [out=0,in=270] ($(rect.north west)!\i!(rect.north east)$);
\foreach \i in {0.1,0.13,...,0.3}
	\draw [red] ($(rect.south west)!\i!(rect.north west)$) to [out=0,in=90] ($(rect.south west)!\i!(rect.south east)$);
\foreach \i in {0.1,0.115,...,0.3}
	\draw [red] ($(rect.north east)!\i!(rect.north west)$) to [out=270,in=180] ($(rect.north east)!\i!(rect.south east)$);
\foreach \i in {0.1,0.12,...,0.3}
	\draw [red] ($(rect.south east)!\i!(rect.south west)$) to [out=90,in=180] ($(rect.south east)!\i!(rect.north east)$);

\foreach \i in {0.4,0.45,0.5,0.55,0.6}
	\draw [red] ($(rect.north west)!\i!(rect.south west)$) to [out=0,in=180] ($(rect.north east)!\i!(rect.south east)$);

\draw [decorate,decoration={brace,amplitude=5pt}] ($(rect.north west)!0.05!(rect.north east)$) -- ($(rect.north west)!0.35!(rect.north east)$) node [midway,yshift=5mm] {$X$};
\draw [decorate,decoration={brace,amplitude=5pt}] ($(rect.north west)!0.65!(rect.north east)$) -- ($(rect.north west)!0.95!(rect.north east)$) node [midway,yshift=5mm] {$W$};
\draw [decorate,decoration={brace,amplitude=5pt}] ($(rect.south east)!0.05!(rect.south west)$) -- ($(rect.south east)!0.35!(rect.south west)$) node [midway,yshift=-5mm] {$Z$};
\draw [decorate,decoration={brace,amplitude=5pt}] ($(rect.south east)!0.65!(rect.south west)$) -- ($(rect.south east)!0.95!(rect.south west)$) node [midway,yshift=-5mm] {$Y$};

\node at (rect.center) [above] {\contour*{white}{$e$}};

\end{tikzpicture}

%% file: tikz_non_flip.tex
\begin{tikzpicture}[scale=2,thick]

\tikzset{dot/.style={shape=circle,fill=black,scale=0.3}}
\tikzset{largedot/.style={shape=circle,fill=black,scale=0.5}}

\coordinate (B) at (0, 0); \coordinate (C) at (0, 1); \coordinate (T) at (0,2);
\coordinate (TL) at (-1,2); \coordinate (TR) at (1,2);

\draw (B) -- (C);
\draw (B) to [out=0,in=0] (TR) to [out=180,in=0] (TL) to [out=180,in=180] (B);

\foreach \i in {0.1,0.15,...,0.9}
	\draw [red]
		($(T)!\i!(TR)$)
		to [out=270,in=0] ($(C)!\i!(B)$)
		to [out=180,in=270] ($(T)!\i!(TL)$);

\node [largedot] at (B) {};
\node [largedot] at (C) {};

\node [dot] at ($(C)!0.2!(B)$) [label=right:{\contour*{white}{$p$}}] {};
\node [dot] at ($(T)!0.2!(TL)$) [label=above:{$p'$}] {};

\node [left] at ($(B)!0.5!(C)$) {\contour*{white}{$e$}};
\node [above] at (T) {$e'$};

\end{tikzpicture}

%% file: tikz_partition.tex
\begin{tikzpicture}[scale=2,thick]

\tikzset{dot/.style={shape=circle,fill=black,scale=0.3}}

\coordinate (L) at (-1, 0); \coordinate (R) at (1, 0);
\coordinate (L2) at (-1, 1); \coordinate (R2) at (1, 1);
\coordinate (L3) at (1, 3); \coordinate (R3) at (1, 1);
\coordinate (L4) at (3, 1); \coordinate (R4) at (1, 1);
\coordinate (L5) at (1, -1); \coordinate (R5) at (1, 1);
\coordinate (L6) at (-1, 1); \coordinate (R6) at (1, 1);
\draw (L) -- (R);
\draw (L2) -- (R2);
\node [left] at (L) {$\emax$};
\node [left] at (L2) {$e$};

\foreach \i in {0.3,0.32,...,0.6}
	\draw [red]
		let \n1 = {\i+0.08} in
		let \n2 = {\i+0.16} in
		let \n3 = {\i+0.34} in
		($(L)!\i!(R)$)
		to [out=90,in=270] ($(L2)!\i!(R2)$)
		to [out=90,in=180] ($(L3)!\n1!(R3)$)
		to [out=0,in=90] ($(L4)!\n2!(R4)$)
		to [out=270,in=0] ($(L5)!\n3!(R5)$)
		to [out=180,in=270] ($(L6)!\n3!(R6)$);

\node [dot] at ($(L)!0.5!(R)$) [label={\contour*{white}{$q$}}] {};
\node [dot] at ($(L2)!0.5!(R2)$) [label={\contour*{white}{$p_i$}}] {};
\node [dot] at ($(L6)!0.84!(R6)$) [label={\contour*{white}{$p_j$}}] {};
\draw [decorate,decoration={brace,amplitude=5pt,mirror}] ($(L)!0.28!(R)$) -- ($(L)!0.6!(R)$) node [midway,yshift=-5mm] {$Q_n$};
\draw [decorate,decoration={brace,amplitude=5pt,mirror}] ($(L2)!0.52!(R2)$) -- ($(L2)!0.82!(R2)$) node [midway,yshift=-5mm] {\contour*{white}{$d$}};

\end{tikzpicture}